\documentclass[12pt]{amsart}
\usepackage{amssymb, latexsym, amscd}
\usepackage[all]{xy}

\theoremstyle{plain}
\newtheorem{theorem}{Theorem}[section]
\newtheorem{corollary}[theorem]{Corollary}
\newtheorem{proposition}[theorem]{Proposition}
\newtheorem{lemma}[theorem]{Lemma}

\theoremstyle{definition}
\newtheorem{definition}[theorem]{Definition}
\newtheorem{remark}[theorem]{Remark}
\newtheorem{example}[theorem]{Example}

\numberwithin{equation}{section}

\usepackage[all]{xy}

\begin{document}

\title[Quasi-Ordered Abelian Groups]{Quasi-Ordered Abelian Groups}

\author[E. Stines]{Elijah Stines}
\address{Department of Mathematics\\
         Iowa State University\\
         Ames, Iowa 50011, U.S.A.}
\email{ejstines@iastate.edu} \urladdr{http://www.public.iastate.edu/~ejstines/}

\begin{abstract}
    Abelian groups having partial orderings compatible with their binary operations have long been studied in the literature. In particular, lattice-ordered abelian groups constitute a universal-algebraic variety, and thus form a category which is monadic over the category of sets. The current paper studies the more general case of quasi-ordered abelian groups, identifying some of their more fundamental properties and their relationships to partially ordered and lattice-ordered groups. We reinterpret the category of quasi-ordered abelian groups with order preserving morphisms by examining the interplay between the group and the set of all positive elements under the quasi-ordering. The main result shows that the category of quasi-ordered abelian groups is monadic over the category of set monomorphisms.
\end{abstract}

\maketitle


\section{Introduction}\label{S:Introduction}
     Partially ordered algebras and lattice-ordered algebras have a long history in the literature. They play a prominent role in the work of G.~Birkhoff \cite{MR598630} and L.~Fuchs \cite{MR0171864}. The special case of lattice-ordered groups (linearly ordered groups above all) has been a particularly well studied area \cite{MR0491396, MR1036072, MR0364051, MR0181678}. More recently there have been some attempts to make a general theory of partially ordered algebras that are not necessarily lattice-ordered \cite{MR598630, MR0171864, Pigozzi}.

    Of particular interest to many are the partially ordered groups, more notably abelian groups, occuring for instance as the additive group reducts of subfields of
    $\mathbb{R}$. The collection of all partially ordered abelian groups forms a category whose morphisms are the isotone group homomorphisms.
    Since the objects of this category are not uniquely defined by their underlying sets and the operations
    on those sets, this category does not form a variety of algebras, and is not monadic over the category of sets.

    Compared to their partially ordered counterparts, abelian groups that are quasi-ordered have received much less attention, although they did appear in Pigozzi's general treatment of quasi-ordered algebras \cite{Pigozzi}.
    One major source of quasi-ordered abelian groups is the divisibility quasi-ordering from ring theory. Recall that, in an integral domain $R$, we say that $a\mid b$ for $a$ and $b \in R\setminus\{0\}$ if and
    only if there exists $r\in R$ so that $ar=b$. The following are standard properties of the relation $|$ of divisibility.

    \begin{proposition}\label{P:qdivis}
        Given an integral domain $R$, we have that $|$ is a transitive and reflexive relation. Furthermore, $|$ is translation invariant, meaning if $a\mid b$ and $c\in R\setminus\{0\}$, then $ac\mid bc$.
    \end{proposition}

    \begin{proof}
        Suppose that $a\mid b$ and $b\mid c$ and that $ar_1=b$ and $br_2=c$, then $ar_1r_2=c$ so $a\mid c$, also $a\mid a$ since $a1=a$. Let $a\mid b$ and $c\in
        R\setminus\{0\}$, then $ar=b$ implies $acr=bc$, so $ac\mid bc$.
    \end{proof}

    At first examination Proposition \ref{P:qdivis} only provides one with a way of constructing quasi-ordered commutative monoids which are important in
    their own right. The true construction of interest for those concerned with factorization or divisibility properties is realized when considering the group of nonzero elements of the quotient
    field of an integral domain $R$. If we extend the quasi-order provided by Proposition \ref{P:qdivis} to the whole group, defining $a\mid b$ in the quotient field
    if and only if ${b}/{a}\in R$, we see that we have a wide array of abelian quasi-ordered groups, that are not partially ordered groups,
    arising from quasi-ordered groups of divisibility. Many times one needs to investigate the orbits of the action of the units of the ring on the nonzero elements of the quotient field. This construction is the classical group of divisibility, discussed at length in \cite{MR720862}. The group of divisibility was seen as a fundamental way to unite factorization theory and the theory of partially ordered abelian groups in \cite{MR1460771}. Partially ordered abelian groups were also examined recently in the investigation of generalized pseudo effect algebras \cite{MR1922915, MR2411097}.

    In the current paper, quasi-ordered abelian groups are used to provide a general framework for the study of various algebraic properties. Specifically, we examine the relationships between lattice-ordered abelian groups, partially ordered abelian groups, and quasi-ordered abelian groups by finding adjoints for each of the forgetful functors between them. We also develop an adjunction between (a category isomorphic to) the category of quasi-ordered abelian groups and a new base category, investigating the free algebras created by the adjunction. The base category is the category of set monomorphisms. The main result shows that the category of quasi-ordered abelian groups is monadic over the category of set monomorphisms (Theorem~\ref{T:QAb2Setm}).

\section{Ordered Groups}\label{S:OrderedGroups}

    We begin by considering some well known facts about lattice-ordered, partially ordered, and quasi-ordered abelian groups and the relationships between them. Several common categories of algebras appear throughout the paper. We record them as follows.

    \begin{definition}
        An algebra $A$ is said to be \emph{partially} (or respectively \emph{quasi-}) \emph{ordered} if there is a partial (or quasi-) ordering $\leq$ ($\preceq$) on the set $A$ such that each
        basic operation $\rho$ is order-preserving or order-reversing in each of its arguments. The homomorphisms between such algebras are the algebra homomorphisms which are order-preserving.

    \end{definition}

    The categories of all abelian lattice-ordered groups (\emph{abelian l-groups}), partially ordered groups (\emph{abelian po-groups}), and quasi-ordered groups (\emph{abelian qo-groups}), with order-preserving group homomorphisms between them, are denoted respectively as $\mathbf{LAb}$, $\mathbf{PAb}$, and $\mathbf{QAb}$.

    There is an evident forgetful functor $U:\mathbf{QAb}\rightarrow \mathbf{Set}$ which forgets all of the structure on the set. There is a left adjoint to this functor as follows. Let $F:\mathbf{Set}\rightarrow \mathbf{QAb}$ be the functor which assigns to a set $X$ the free abelian group on $X$ equipped with the discrete or trivial ordering $x\leq y$ if and only if $x=y$.

    \begin{proposition}\label{P:setqogrp}
        The functor $F$ is left adjoint to the forgetful functor $U$.
    \end{proposition}

    \begin{proof}
        Let $G$ be an abelian qo-group. Each set map from $X$ to $UG$ induces a unique group homomorphism $f:FX\rightarrow G$ from the free abelian group $FX$ on $X$. This group homomorphism preserves the trivial ordering.
    \end{proof}

    One can see how the definition of an abelian qo-group is not entirely algebraic in nature. Requiring the algebra homomorphisms to satisfy an additional property, such as order preservation, is not part of traditional universal algebra. In the course of this discussion we shall reinterpret $\mathbf{PAb}$
    and $\mathbf{QAb}$ into more algebraic terms. We are in need of an efficient way to discern the partial order structure on one of these objects via algebraic conditions. We begin by defining the positive cone of an abelian qo-group.

    \begin{definition}
        For an abelian qo-group $(G,\preceq)$, we define $$G^+_0:=\{ g \in G : 0\preceq g\}$$ to be the \emph{positive cone} of $(G,\preceq)$.
    \end{definition}

    We now characterize those subsets of abelian groups which may be considered as positive cones with respect to some quasi-ordering or partial ordering. The group of units of a monoid $M$ is denoted by $M^*$.

    \begin{definition}
                  A cancellative monoid $(M,+,0)$ is said to be \emph{conical} if $M^*=\{0\}$.
    \end{definition}

    \begin{proposition}\label{P:ConeCharacterization}
        Let $G$ be an abelian group and $M\subseteq G$. Then:

        \begin{enumerate}

 \item[$(\mathrm{a})$]       $M$ is the positive cone of a qo-group $(G,\preceq)$ if and only if $M$ is a monoid.

 \item[$(\mathrm{b})$]       $M$ is the positive cone of a po-group $(G, \leq)$ if and only if $M$ is a conical monoid.

    \end{enumerate}

    \end{proposition}

    \begin{proof}
        Suppose that $M$ is a submonoid of $G$. Define the quasi-order as $a\preceq b$ if and only if $b-a \in M$. We have that $a-a \in M$ for $a \in G$, thus $\preceq$ is reflexive. Furthermore if $a\preceq b$ and $b\preceq c$, then $c-a= (c-b)+(b-a)\in M$. Finally if $a\preceq b$ and $c\in G$, we have $(b-c)+(c-a)=b-a\in M$. Now suppose that $M$ is the positive cone for an abelian qo-group $(G,\preceq)$. We must have $0\in M$ by reflexivity, and if $a$,$b\in M$ so $0\preceq a$ and $0\preceq b$. Now $b\preceq a+b$ by translation invariance, so $0 \preceq b \preceq a+b$ implies $0\preceq a+b$ by transitivity.

        In the special case where $M$ is a conical submonoid of $G$, we must have that if $b-a$, $a-b \in M$ then $(b-a)+(a-b)=0$, so $a-b=0$ and $a=b$. Therefore the quasi-ordering created is, in fact, a partial ordering. On the other hand, if $(G, \leq)$ is an abelian po-group, then if $0\leq a$ and $0\leq -a$, we must have $a\leq 0 \leq a$, so $a=0$. Finally, since $M$ is a submonoid of an abelian group $G$, it is necessarily canellative.
    \end{proof}

    \begin{proposition}
        Given an abelian group $G$ with a submonoid $M$ of $G$, we have that $M$ is the positive cone of an abelian po-group $(G,\leq)$ if and only if $M\setminus \{0\}$ is a subsemigroup of $G$.
    \end{proposition}

    \begin{proof}
        Suppose that $M$ is the positive cone of an abelian po-group $(G,\leq)$. From Proposition \ref{P:ConeCharacterization} we know that $M$ is a conical monoid. Since $M$ is a monoid, it is closed under the binary operation, and since $M$ is conical, the set of nonzero elements is also closed under the binary operation. Thus $M\setminus\{0\}$ is a semigroup.

        On the other hand, if $M\setminus\{0\}$ is a subsemigroup of $G$, we have that $M$ is a conical monoid. Thus, there is a partial order $\leq$ such that $M$ is the positive cone of $(G, \leq)$.
    \end{proof}

    Examples of abelian po-groups abound in mathematics, and are easily recognizable to most. An explicit source of examples for abelian qo-groups should provide a better understanding of the difference between these two related concepts. In view of Proposition \ref{P:ConeCharacterization}, we may construct
    abelian qo-groups as follows.

    \begin{example}\label{E:finite}
        Given a finite abelian group $G$ and a nonzero submonoid  $M$ of $G$, the monoid $M$ is the positive cone of a qo-group $(G,\preceq)$ which is not an abelian
        po-group. This is because each element of $M$ must have torsion, as $M$ is finite.
    \end{example}

    In the same way that one may define an equivalence relation on a quasi-ordered set identifying all element pairs $a,b$ with $a\preceq b$ and $b\preceq a$, one may do the same for abelian qo-groups. The process of taking the quotient by this equivalence relation (which is, in fact, a congruence) is called the \emph{antisymmetrization} of the abelian qo-group.

    \begin{lemma}

        Consider an abelian qo-group $(G,\preceq)$ with positive cone $M$. Then for elements $a$, $b$ of $G$, the following are equivalent:

        \begin{enumerate}

        \item[$(\mathrm{a})$] $a \preceq b$ and $b \preceq a$;

        \item[$(\mathrm{b})$] $a+M^*=b+M^*$.

        \end{enumerate}

    \end{lemma}

    \begin{proposition}\label{P:symmetrization}
        Given an abelian qo-group $(G,\preceq)$ with positive cone $M$, the quotient $(G/M^*, \preceq')$, is an abelian po-group with partial ordering
        $a+M^* \preceq' b + M^*$ if and only if $b-a \in M^*$.
    \end{proposition}

    \begin{proof}
       First note that $M^*$ is a (normal) subgroup of $G$, so the quotient $G/M^*$ is an abelian group. It is clear that $M/M^*$ is a submonoid of $G/M^*$. We now verify that $M/M^*$ is conical. Suppose that $a+b+M^*=M^*$ for $a,b \in M$. That is, $a+b \in M^*$. Then $-b-a \in M^*\subset M$, so $-b-a+a=-b\in M$, whence $b \in M^*$ and $(G/M^*,\preceq')$ is an abelian po-group.
    \end{proof}

    In light of Example \ref{E:finite}, the antisymmetrization of a finite abelian qo-group is trivially ordered. We also have that antisymmetrization is idempotent.

    The discussion of the positive cone for an abelian po-group reveals an intrinsic connection between the two. Since the morphisms $f$ of $\mathbf{QAb}$ are order-preserving, we have that $0\preceq f(a)$ for any element $a$ of the positive cone. So we see that the order-preserving morphisms are precisely
    the group homomorphisms mapping positive cones into positive cones.

    \begin{definition}
        Let $$C: \mathbf{QAb} \rightarrow \mathbf{Mon}$$ be the functor from abelian qo-groups to monoids which assigns to each abelian qo-group its positive cone, and to each order-preserving morphism, the corresponding monoid homomorphism between the positive cones.
        Let $$U': \mathbf{QAb} \rightarrow \mathbf{Mon}$$ be the forgetful functor which forgets both the inverses and quasi-ordering.
        Finally, let $$\tau: C\rightarrow U'$$ be the natural transformation between these functors with components as in Figure \ref{F:transformation}.
    \end{definition}

    \begin{figure}

        $$\xymatrix{ G\ar[d]^f & G^+_0\ar[d]^{Cf}\ar[r]^{\tau_G} & U'G\ar[d]^{U'f} \\ H & H^+_0\ar[r]^{\tau_H} & U'H}$$

    \caption{A morphism of $\mathbf{QAb}^i$}\label{F:transformation}
    \end{figure}

    The components of $\tau$ at an abelian qo-group will be the objects of a category $\mathbf{QAb}^i$ with vertical composition of the maps
    as described in \cite{MR1712872}, where each map is induced by an order-preserving homomorphism $f:G\rightarrow H$ as in Figure \ref{F:transformation}.
    The most important observation is that $\mathbf{QAb}^i$ is isomorphic to $\mathbf{QAb}$ itself.

    \begin{theorem}\label{T:qabqabi}
        There is an isomorphism of categories between $\mathbf{QAb}$ and $\mathbf{QAb}^i$.
    \end{theorem}

    \begin{proof}
        Consider the functor $A:\mathbf{QAb}\rightarrow \mathbf{QAb}^i$ which assigns, to each order-preserving morphism on the left of Figure \ref{F:transformation}, the commuting diagram on the right. Further, consider the functor $B:\mathbf{QAb}^i\rightarrow \mathbf{QAb}$, which assigns, to any component of the natural transformation $\tau_G$, an abelian qo-group $G$ with positive cone $C(G)$. The functor $B$ sends the commuting diagram on the right of Figure \ref{F:transformation} to the morphism on the left.
        It is readily observed that $AB=1_{\mathbf{QAb}^i}$ and $BA=1_{\mathbf{QAb}}$. Therefore the two categories are isomorphic.
    \end{proof}

     A component $\tau_G$ of the natural transformation $\tau$ can be considered as the function that inserts the positive cone of the abelian qo-group $G$ into the group. We see that instead of taking the group as the main object of interest in our study of abelian qo-groups, we may instead concentrate on the monoid insertion $\tau_G$. This technique will be used to establish a monadic relationship between abelian qo-groups and a base category in Section~\ref{S:Adjunction}.

\section{Set monomorphisms}\label{S:BaseCategory}

    In Section \ref{S:OrderedGroups} we examined some of the fundamental properties of ordered abelian groups, and reexamined the relationship between the positive cone of an abelian qo-group and the group itself. Specifically, we saw that the category $\mathbf{QAb}$ is isomorphic to the category $\mathbf{QAb}^i$, the category of components of the natural transformation $\tau:C\rightarrow U'$ with vertical composition of the induced maps. It is now possible to relate the category $\mathbf{QAb}^i$ to suitable base categories, the categories of set insertions and set monomorphisms. Further, we develop a more general category $\mathbf{QAb}^m$ which will be seen to be equivalent to $\mathbf{QAb}^i$.

    \begin{definition}\label{D:base}
        The objects of the \emph{category} $\mathbf{Set}^i$ \emph{of set insertions} are the insertions $i:X'\hookrightarrow X$ of subsets into supersets. The morphisms of the category are pairs of set maps
        $$(f_1, f_2):(i:X'\hookrightarrow X)\rightarrow (j:Y'\hookrightarrow Y)$$
        where $f_1:X\rightarrow Y$, $f_2:X'\rightarrow Y'$, and
        $j\circ f_2=f_1\circ i$. The composition is clearly associative by the inspection and verification of the commuting squares of Figure~2
        and regarding $f\circ g=(f_1\circ g_1,f_2\circ g_2)$.
    \end{definition}

    \begin{figure}[hbt]
        $$\xymatrix{X'\ar@{^(->}[r]^{i}\ar[d]_{f_2} & X\ar[d]^{f_1} \\ Y'\ar@{^(->}[r]^{j} & Y}$$
        \caption{A morphism of $\mathbf{Set}^i$}
    \end{figure}\label{F:setmorph}

    In Definition \ref{D:base} we have set insertions as the objects of the category. In the current form it would be acceptable to view each of the objects as a pair, say $(X',X)$ with no real confusion as to how $X'$ is identified inside of the superset $X$. It is possible to generalize the situation by defining an equivalent related category.

    \begin{definition}\label{D:ibase}
        The objects of the \emph{category} $\mathbf{Set}^m$ \emph{of set monomorphisms} are injective set maps $i:X'\rightarrow X$. The morphisms of the category are pairs of set maps
        $$(f_1, f_2):(i:X'\rightarrow X)\rightarrow (j:Y'\rightarrow Y)$$
        where $f_1:X\rightarrow Y$, $f_2:X'\rightarrow Y'$, and
        $j\circ f_2=f_1\circ i$. Composition of morphisms in $\mathbf{Set}^m$ is given by the same ``componentwise'' composition as defined for $\mathbf{Set}^i$.
    \end{definition}

    \begin{theorem}\label{T:setequiv}
        There is an equivalence of categories between $\mathbf{Set}^i$ and $\mathbf{Set}^m$.
    \end{theorem}

    \begin{proof}
        Define a functor $$A:\mathbf{Set}^i\rightarrow \mathbf{Set}^m$$ on the morphisms as $A(i:X'\hookrightarrow X)=Ai:X'\rightarrow X$, which forgets that set insertions are insertions and simply considers them as set monomorphisms. Define $B:\mathbf{Set}^m\rightarrow \mathbf{Set}^i$ as $B(i:X'\rightarrow X)=Bi: i(X')\hookrightarrow X$, which is the insertion of the image of $i$ into the set $X$.

        We must show that $A$ and $B$ constitute an equivalence of categories. We show that $A$ is full, faithful, and dense in $\mathbf{Set}^m$. Suppose $f,g\in \mathbf{Set}^i(i,j)$ with $f\not= g$. We have that $f_1$ and $g_1$ must differ on at least one element of $X$, the codomain of $i$. This is because $f$ and $g$ are uniquely determined by their definition on $X$, as described in \cite{MR2377903}. Then, since $Af$ and $Ag$ return the same maps $f_1$ and $g_1$ on $X$, we can be sure that they are not the same.

        Now consider $f \in {\mathbf{Set}^m}(Ai,Aj)$. The morphism $f$ is defined by its two components $f_1$ on $X$ and $f_2$ on $X'$. Since $i$ and $j$ are monomorphisms, they are injective. Thus the definition of $f_1$ uniquely determines the definition of $f_2$ via $j\circ f_2= f_1 \circ i$. Therefore $f=Af'$ for some $f'\in  {\mathbf{Set}^m}(i,j)$.

        Finally, we have that $A$ is dense since, given any object $i:X'\rightarrow X$, we have that $i\cong ABi$, by virtue of the morphism induced from the identity map on $X$.
    \end{proof}

    Recall that in $\mathbf{QAb}^i$ we may think of each of the components of $\tau$ as the insertion map of the positive cone $G^+_0$ into the group $G$. While this does provide us with quite a succinct way to view $\mathbf{QAb}$ through an isomorphic category, we may use the principle outlined in Theorem~\ref{T:setequiv} to define a category $\mathbf{QAb}^m$ equivalent to $\mathbf{QAb}^i$.

    \begin{figure}[hbt]
        $$\xymatrix{M\ar[r]^{i}\ar[d]_{f_2} & U'G\ar[d]^{U'f_1} \\ N\ar[r]^{j} & U'H}$$
        \caption{A morphism of $\mathbf{QAb}^m$}
    \end{figure}\label{F:qabmmorphism}

    \begin{definition}
        Let the objects of $\mathbf{QAb}^m$ be monoid monomorphisms $i: M \rightarrow U'G$, where $U'G$ is the monoid reduct of an abelian group $G$. Let the morphisms of $\mathbf{QAb}^m$ be commuting squares as in Figure~\ref{F:qabmmorphism}, induced by a group homomorphism $f:G\rightarrow H$ with componentwise composition similar to that of $\mathbf{Set}^m$.
    \end{definition}

    \begin{theorem}\label{T:qabequiv}
        There is an equivalence of categories between $\mathbf{QAb}^i$ and $\mathbf{QAb}^m$.
    \end{theorem}

    \begin{proof}
        Define $A:\mathbf{QAb}^i\rightarrow \mathbf{QAb}^m$ to be the functor that forgets that Figure~\ref{F:transformation} came from a natural transformation, and simply considers it as a diagram in $\mathbf{Mon}$. Define $B:\mathbf{QAb}^m\rightarrow \mathbf{QAb}^i$ to be the functor that assigns, to each monoid monomorphism $i:M\rightarrow U'G$, the component $\tau_G$ where $G$ is the abelian qo-group with positive cone $i(M)\subseteq G$. The morphism part of $B$ simply assigns the morphism of $\mathbf{QAb}^i$ corresponding to the abelian qo-group morphism induced by $f:G\rightarrow H$.
        The fact that $A$ and $B$ form an equivalence is simply a special case of Theorem \ref{T:setequiv}.
    \end{proof}

    In defining $\mathbf{QAb}^m$ as we have, we have moved the emphasis from the actual group $G$ and the submonoid $M$ of the positive cone to the way in which the elements of some monoid $M$ are identified inside of $G$. In particular, the existence of a monoid monomorphism $i:M\rightarrow U'G$ implies that $M$ is commutative and cancellative. The focus on the identification provides a sufficiently abstracted framework for us to discuss a monadic adjunction between $\mathbf{QAb}^m$ and $\mathbf{Set}^m$ in Section \ref{S:Adjunction}. Before we describe such an adjunction, let us examine some specific objects of $\mathbf{QAb}^m$.

    \begin{example}\label{E:qogroups}
    In each of the following cases, we consider a qo-group structure on $\mathbb{Z}$.

    \begin{enumerate}

        \item We may identify the abelian l-group of $\mathbb{Z}$ under the usual order as the object $i:\mathbb{N}\rightarrow U'\mathbb{Z}$, where $i$ is the usual inclusion.

        \item Let $2:\mathbb{N}\rightarrow U'\mathbb{Z}$ be the map that doubles every element and inserts it into the group $\mathbb{Z}$. We obtain the abelian po-group $(\mathbb{Z},\le)$ with
            $m\leq n$ if and only if $n-m =2k$ for $k\in \mathbb{N}$.

        \item Letting $2:\mathbb{Z}\rightarrow \mathbb{Z}$ be a doubling map again, the resulting abelian qo-group is $(\mathbb{Z},\preceq)$, where $m\preceq n$ if and only if $n-m =2k$ for $k\in \mathbb{Z}$.

    \end{enumerate}
    \end{example}

    Example~\ref{E:qogroups} shows that it is not the  positive cone monoid itself that is important for the structure of an abelian qo-group, but rather the way in which that cone is identified inside the group. This is why we take monomorphisms as the objects, and not just the pairs $(M,G)$.

    In the special case of abelian po-groups, we define categories corresponding to $\mathbf{QAb}^i$ and $\mathbf{QAb}^m$.

    \begin{definition}
        The category $\mathbf{PAb}^i$ is the full subcategory of $\mathbf{QAb}^i$ where all the positive cones are positive cones of abelian po-groups. The category $\mathbf{PAb}^m$ is the full subcategory of $\mathbf{QAb}^m$ where all the monoids are conical.
    \end{definition}

    \begin{proposition}\label{P:pabisomom}
        $\mathbf{PAb}^i$ is equivalent to $\mathbf{PAb}^m$ and isomorphic to $\mathbf{PAb}$.
    \end{proposition}

    \begin{proof}
        Recall the functors $A$ and $B$ in Theorem~\ref{T:qabequiv}. Take the restrictions of the functors to the respective subcategories $\mathbf{PAb}^i$ and $\mathbf{PAb}^m$. For the second statement, take the restrictions of the functors of Theorem~\ref{T:qabqabi} to the respective subcategories $\mathbf{PAb}^i$ and $\mathbf{PAb}^m$.
    \end{proof}

    In defining $\mathbf{QAb}^i$ and $\mathbf{PAb}^i$, we identified the properties that a monoid must have in order for it to be the positive cone of an abelian po-group or an abelian qo-group. Namely, we saw that an arbitrary submonoid of a group could be the positive cone of an abelian qo-group, and an arbitrary conical monoid could be the positive cone of an abelian po-group. Example~\ref{E:qogroups} shows that there is no comparable characterization of the monoids which may be the positive cone of an abelian l-group, since Example~\ref{E:qogroups}(1) and Example~\ref{E:qogroups}(2) have isomorphic monoids as their positive cones, while the first example is an l-group and the second is not. In order to define categories $\mathbf{LAb}^i$ and $\mathbf{LAb}^m$ corresponding to $\mathbf{QAb}^i$ and $\mathbf{QAb}^m$ for lattice-ordered abelian groups, we thus use extrinsic properties rather than an internal definition relying on positive cones.

    \begin{definition}
        Let $\mathbf{LAb}^i$ be the category whose objects are insertions $i: G_0^+\hookrightarrow G$ of the positive cone of an abelian l-group into the abelian l-group, and whose morphisms are pairs $f=(f_1,f_2)$, where $f_1$ is an abelian l-group morphism and $f_2$ is the restriction of $f_1$ to the positive cone.

        Similarly, define $\mathbf{LAb}^m$ be the category whose objects are monoid monomorphisms $i: M\rightarrow G$ of $M$ onto the positive cone of an abelian l-group $G$. The morphisms of $\mathbf{LAb}^m$ are pairs $f=(f_1, f_2)$, where $f_1$ is an abelian l-group morphism and $f_2$ is a monoid morphism, with $j\circ f_2= f_1 \circ i$.
    \end{definition}

    \begin{figure}
        $$\xymatrix{
        \mathbf{PAb}\ar[r]^{U_{PQ}} & \mathbf{QAb}\ar[r]^{U_Q} & \mathbf{Set} & \mathbf{LAb}\ar[l]_{U_L}\\ \mathbf{PAb}^i\ar[r]^{U_{PQ}^i}\ar[d]^{U_P^{im}}\ar[u]_{U_P^i} & \mathbf{QAb}^i\ar[u]_{U_{QQ}^i}\ar[r]^{U_Q^i}\ar[d]^{U_Q^{im}} & \mathbf{Set}^i\ar[u]_{U^i}\ar[d]^{U^{im}} & \mathbf{LAb}^i\ar[u]_{U_{LL}^i}\ar[d]^{U_L^{im}}\ar[l]_{U_{L}^i}  \\ \mathbf{PAb}^m\ar[r]^{U_{PQ}^m} & \mathbf{QAb}^m\ar[r]^{U_Q^m} & \mathbf{Set}^m & \mathbf{LAb}^m\ar[l]_{U_{L}^m}}  $$
        \caption{Forgetful Functors}\label{F:forgetfuldiagram}
    \end{figure}

    \begin{remark}[Notational conventions]
        In the rest of this section, a functor denoted by a decorated $U$ will be a forgetful functor. A left adjoint to one of these will be denoted by $A$ with the same subscripts and superscripts. The subscripts $P$, $L$, and $Q$ denote the order structure on the domain and codomain of the forgetful functor. If there is only one letter in the subscript, either the codomain is related to $\mathbf{Set}$ or the order type is not changing. The superscripts $i$ and $m$ denote insertions or monomorphisms for the domain and codomain. If there is only one superscript, the functor disregards the insertion.
    \end{remark}

    There is a collection of evident forgetful functors between all of the categories discussed so far, as recorded in Figure~\ref{F:forgetfuldiagram}. We have already seen some of the left adjoints to these functors, and some are particularly well known. We shall conclude this section by indicating all of the adjoints to these forgetful functors. The forgetful functors $U_L^i$ and $U_L^m$ have not yet been defined, but will also be seen as right adjoints by the end of this section.

    \begin{corollary}\label{C:doneadjunctions}
         Consider the forgetful functors of Figure~\ref{F:forgetfuldiagram}. Then $U_{LL}^i, U_{P}^i,$ and $U_{QQ}^i$ are all isomorphisms of categories. Furthermore, the forgetful functors $U_L^{im}, U_P^{im}, U_Q^{im},$ and $U^{im}$ provide equivalences of categories.
    \end{corollary}

    \begin{proof}
        The first assertion is a summary of the statements of Proposition~\ref{P:pabisomom} and Theorem~\ref{T:qabqabi}, and the special case of the restriction of the isomorphisms to $\mathbf{LAb}$ and $\mathbf{LAb}^i$.

       We have that $U_Q^{im}$ is an equivalence by Theorem~\ref{T:qabequiv}, the functor $U_P^{im}$ is an equivalence by the first portion of Proposition~\ref{P:pabisomom}, and the fact that $U^{im}$ is an equivalence results from Theorem~\ref{T:setequiv}. Finally, the fact that $U_L^{im}$ is an equivalence follows by restriction from $\mathbf{PAb}$ to $\mathbf{LAb}$.
    \end{proof}

    \begin{proposition}
        Each of the forgetful functors $U_Q, U^i, U_Q^i,$ and $U_Q^m$ of Figure~\ref{F:forgetfuldiagram} is a right adjoint.
    \end{proposition}

    \begin{proof}
        Let us begin with $U^i$. Consider the functor $A^i:\mathbf{Set}\rightarrow \mathbf{Set}^i$ with morphism part which sends $f:X\rightarrow Y$ to $A^if=(f, 0)$, where $0:\emptyset\hookrightarrow Y'$. We may see that $U^iA^i$ is the identity functor, so the unit of the adjunction is the identity natural transformation. Similarly $F^iA^i$ sends the morphism $(f_1,f_2)$ to the morphism $(f_1, 0)$, so the component of the counit of the adjunction at an object $X'\rightarrow X$ is the pair consisting of the identity map on $X$ and the empty map $\emptyset\hookrightarrow X'$.

        We now show that $U_Q$ is the right adjoint of an adjunction. The fact that $U_Q^i$ and $U_{Q}^m$ are right adjoints then follows from a restatement of the definition of the left adjoint to $U_Q$. Define $A_Q:\mathbf{Set}\rightarrow \mathbf{QAb}$ to have morphism part sending $f:X\rightarrow Y$ to $\overline{f}:F_1X\rightarrow F_1Y$, where $\overline{f}$ is the induced map from the free abelian group $F_1X$ to the free abelian group $F_1Y$, in which both $F_1X$ and $F_1Y$ are trivially ordered. Then, by definition, all group homomorphisms from $F_1X$ to an abelian qo-group $G$ are necessarily order-preserving, which means we may use the restrictions of the usual unit and counit components from the adjunction between $\mathbf{Set}$ and $\mathbf{Ab}$.
    \end{proof}

    \begin{proposition}
        Each of the forgetful functors $U_{PQ}, U_{PQ}^i,$ and $U_{PQ}^m$ of Figure~\ref{F:forgetfuldiagram} is a right adjoint.
    \end{proposition}

    \begin{proof}
        Define the functor $A_{PQ}:\mathbf{QAb}\rightarrow \mathbf{PAb}$ to have a morphism part that takes $f:G\rightarrow H$ to $$f/G_0^{+*}: G/G_0^{+*}\rightarrow H/H_0^{+*}; g+G_0^{+*}\mapsto f(g)+H_0^{+*}\, . $$ This is a well defined group homomorphism, since if $a\in G_0^{+*}$, then $-a\in G_0^{+*}$ by definition. Further, since $f$ is order-preserving, we must have that $f(a)\in H_0^+$ and $f(-a)\in H_0^+$, so $f(a)\in H_0^{+*}$.

        The functor $A_{PQ}U_{PQ}$ is the identity, so $\varepsilon_G: A_{PQ}F_2G=G\rightarrow G$ is the identity map. Also, $\eta_G:G\rightarrow U_{PQ}A_{PQ}G=U_{PQ}G/G_0^{+*}$ is the quotient map. It is clear that
        $$\xymatrix{ A_{PQ}\ar[r]^{A_{PQ}\eta\phantom{mmm}} & A_{PQ}U_{PQ}A_{PQ}\ar[r]^{\phantom{mmm}\varepsilon A_{PQ}} & A_{PQ}}$$
        and
        $$\xymatrix{ U_{PQ}\ar[r]^{\eta U_{PQ}\phantom{mmm}} & U_{PQ}A_{PQ}U_{PQ}\ar[r]^{\phantom{mmm}U_{PQ}\varepsilon} & U_{PQ}}$$
        are identity morphisms.

        If we define $A_{PQ}^i(i:G_0^+\hookrightarrow G)=A_{PQ}^ii: G_0^+/G_0^{+*}\hookrightarrow G/G_0^{+*}$ and extend it to a morphism part in a similar fashion to $A_{PQ}$, we obtain $A_{PQ}^i$ as the left adjoint to $U_{PQ}^i$.

        Finally, the specification of
        $$A_{PQ}^m:(i:M\rightarrow G) \mapsto (A_{PQ}^mi: M/M^*\rightarrow G/i(M^*))$$
        will provide a left adjoint to $U_{PQ}^m$.
    \end{proof}

    The adjoint situations for $U_L:\mathbf{LAb}\rightarrow \mathbf{Set}$, $U_L^i:\mathbf{LAb}^i\rightarrow\mathbf{Set}^i$, and $U_L^i:\mathbf{LAb}^m\rightarrow \mathbf{Set}^m$ are well known and, in fact, the first adjunction is even monadic. What is notable about the adjunction between $\mathbf{LAb}$ and $\mathbf{Set}$ are the free objects. Consider the following result due to Birkhoff \cite{MR598630}.

    \begin{example}\label{E:FreeOnOne}
        Let $X=\{x\}$ be a set. Then the free lattice-ordered group $A_L(X)$ on $X$ is isomorphic to $\mathbb{Z}\oplus \mathbb{Z}$ with the componentwise ordering. The generator for this lattice-ordered group is obtained by mapping $x$ to $(1,-1)$.
    \end{example}

    It was also shown in \cite{MR2318248} that, given any finitely generated (not necessarily free) abelian l-group, the underlying group is a free abelian group. At this point, it is important to note that there is no left adjoint to the forgetful functor from $\mathbf{PAb}$ and $\mathbf{LAb}$. This issue was discussed by Weinberg \cite{MR0153759}, and later refined by Conrad \cite{MR0270992}. The main result from \cite{MR0270992} is as follows. We begin with a definition.

    \begin{definition}
        Given a group $G$ we say that a partial order $\leq$ on $G$ is a \emph{right order} if $\leq$ is a linear order and $ab\leq ac$ whenever $b\leq c$.
    \end{definition}

    For an abelian po-group, the notions of a right order and a linear order are equivalent.

    \begin{theorem}\label{T:noleft}(Conrad, 1970)
        For a po-group $G$ the following are equivalent:

        \begin{enumerate}
            \item There exists a free l-group over $G$;
            \item There exists a $\mathbf{PAb}$ isomorphism between $G$ and the underlying po-group of an l-group;
            \item The positive cone $G_0^+$ of $G$ is an intersection of right orders.
        \end{enumerate}
    \end{theorem}

    We now have the following.

    \begin{corollary}
        The forgetful functor $U:\mathbf{LAb}\rightarrow \mathbf{QAb}$ has no left adjoint.
    \end{corollary}

    \begin{proof}
        Let $T$ be a finitely generated abelian qo-group with non-trivial torsion elements and antichain ordering. Since $A_{PQ}$ takes the antisymmetrization of the qo-group, we have that there is no free abelian l-group over $A_{PQ}T$, since condition (2) of Theorem~\ref{T:noleft} is violated (finitely generated l-groups cannot have nontrivial torsion elements). With this observation, and the fact that adjoints are unique, we have that there is no left adjoint to $U$, as desired.
    \end{proof}

\section{Free algebras
}\label{S:Adjunction}

In the preceding section we discussed the relationships between several categories related to $\mathbf{QAb}$ and summarized the relevant adjunctions between them in Figure~\ref{F:forgetfuldiagram}. Of particular interest at this point is the left adjoint to the functor $U_{Q}^m:\mathbf{QAb}^m\rightarrow \mathbf{Set}^m$ (hereafter abbreviated to $U$), since the images of $\mathbf{Set}^m$-objects under the left adjoint will be regarded as free abelian qo-groups. (According to Proposition~\ref{P:setqogrp}, the images of objects under the left adjoint from $\mathbf{Set}$ to $\mathbf{QAb}$ have trivial order structure, and thus fail to encompass the full gamut expected of free abelian qo-groups.) We restate the situation as follows.

\begin{definition}\label{D:leftadjoint}
    Let $f:(i:X'\rightarrow X)\rightarrow (j:Y'\rightarrow Y)$ be a morphism in $\mathbf{Set}^m$, where $f$ is induced by the set map $f_1:X\rightarrow Y$. Define $Ff:(Fi:F_2X'\rightarrow U'F_1X)\rightarrow (Fj:F_2Y'\rightarrow U'F_1Y)$ to be the $\mathbf{QAb}^m$-morphism determined by the group homomorphism $F_1f:F_1X\rightarrow F_1Y$, where $F_1X$ is the free abelian group on $X$ and $F_2X'$ is the free commutative monoid on $X'$.
\end{definition}

\begin{definition}\label{D:abadjunction}
    Let $U_1:\mathbf{Ab}\rightarrow \mathbf{Set}$ be the forgetful functor, part of the adjunction $(F_1, U_1, \eta_1, \varepsilon_1)$.
\end{definition}

\begin{theorem}\label{T:leftadjoint}
    Let $U$ be the forgetful functor from $\mathbf{QAb}^m$ to $\mathbf{Set}^m$. Consider $F$, $F_1$, $F_2$, and $U'$ as in Definitions~\ref{D:leftadjoint} and \ref{D:abadjunction}. Then $F$ is left adjoint to $U$.
\end{theorem}

\begin{proof}
    We prove the theorem by constructing the unit and counit of the claimed adjunction, and verifying the triangular identities. Define $$\eta_i: (i:X'\rightarrow X)\rightarrow (UFi:U_2F_2X'\rightarrow U_2U'F_2X)\, ,$$ the map induced by the set homomorphism $\eta_{1X}$. Define $$\varepsilon_j: (FUj:F_2U_2M\rightarrow F_1U_2U'G)\rightarrow (j:M \rightarrow U'G)\, ,$$ the map induced by the abelian group homomorphism $\varepsilon_{1G}$.

    We have $\xymatrix{ F_1\ar[r]^{F_1\eta_1\phantom{mm}} &F_1U_1F_1\ar[r]^{\phantom{nm}\varepsilon_1F_1} & F_1}$ as the identity at each set and $\xymatrix{ U_1\ar[r]^{\eta_1 U_1\phantom{mm}} &U_1F_1U_1\ar[r]^{\phantom{m}U_1\varepsilon_1} & U_1}$ as the identity at each abelian group. Furthermore, if $f:(i:X'\rightarrow X)\rightarrow (i:X'\rightarrow X)$ is induced by the identity on $X$, the commuting diagram and the monic nature of $i$ forces $f$ to be the identity morphism. The same holds for a morphism in $\mathbf{QAb}^m$.
\end{proof}

\begin{definition}
The image of an object of $\mathbf{Set}^m$ under $F$ is described as a \emph{free object} of $\mathbf{QAb}^m$, or as a \emph{free abelian qo-group}.
\end{definition}

Recall the characterization of free abelian groups and free commutative monoids.

    \begin{proposition}\label{P:freegroups}
        An abelian group $G$ is a free abelian group if and only if there is an indexing set $X$ so that $G\cong \bigoplus\limits_{x\in X} \mathbb{Z}$. A commutative monoid $M$ is a free commutative monoid if and only if there is an indexing set $X'$ so that $M\cong\bigoplus\limits_{x\in X'} \mathbb{N}$.
    \end{proposition}

    \begin{definition}
        In the event that the indexing set for a free abelian group (or commutative monoid) is finite, the cardinality of $X$ is commonly known as the \emph{rank} of the free abelian group (or commutative monoid). The rank determines a free abelian group (or commutative monoid) up to isomorphism. The function $\textup{rk}\,(G)$ assigns, to each free abelian group (or commutative monoid) $G$, its rank.
    \end{definition}

In light of Proposition~\ref{P:freegroups}, we may now characterize free abelian qo-groups in the context of the base category $\mathbf{Set}^m$ as follows.

    \begin{theorem}
        An abelian qo-group associated with $i:M\rightarrow G$ is a free abelian qo-group if and only if $M$ and $G$ are free.
    \end{theorem}

    \begin{proposition}
        If $i: M \rightarrow G$ is free, then $\textup{rk}\,(M)\leq \textup{rk}\,(G)$.
    \end{proposition}

    Since the positive cone monoid determines the order structure on an abelian qo-group, we may observe the following.

    \begin{proposition}
        Any free abelian qo-group $i:M \rightarrow G$ can be viewed, in terms of $\mathbf{QAb}$, as $(\bigoplus\limits_{x\in X'}
        (\mathbb{Z},\leq))\oplus(\bigoplus\limits_{x\not\in X'} (\mathbb{Z}, =) )$, where the first term is a sum of copies of $\mathbb{Z}$ with the usual order, while the
        second term is a sum of $\mathbb{Z}$ with the discrete order.
    \end{proposition}

    \begin{proof}
        We translate the expression of the group in terms of direct sums into an expression of the insertion of a positive cone in $\mathbf{QAb}^i$. We recognize
        that in terms of the group structure we have $\bigoplus\limits_{x\in X} \mathbb{Z}$, and that an element of this group is in the positive
        cone if and only if $a_x\geq 0$ for each term in the sum. This means that $a_x=0$ for each $x\not\in X'$. Thus we may recognize the submonoid of the positive cone as $\bigoplus\limits_{x\in X'} \mathbb{Z}$ and take the monoid monomorphism as the natural set insertion.
    \end{proof}

    \begin{proposition}
        Suppose that $i:M\rightarrow G$ is a free abelian qo-group over $i':X'\to X$. Then:

        \begin{enumerate}
            \item $M$ is a conical monoid;
            \item $i$ is associated with an abelian qo-group;
            \item $i$ is associated with a lattice if and only if $i'$ is an isomorphism;
            \item $i$ is associated with a chain if and only if $i'$ is an isomorphism and $X=\{x\}$;
            \item $i$ is associated with an antichain ordering if and only if $X'=\emptyset$.
     \end{enumerate}
    \end{proposition}

    \begin{proposition}
        There are $n+1$ isomorphism classes for free abelian qo-groups whose abelian group reducts are free groups of rank $n$.
    \end{proposition}

    \begin{proof}
        Since every free abelian qo-group is a direct sum of copies of $\mathbb{Z}$, we may arrange the sums so that the terms corresponding to the
        image of $i'$ come first. Thus two free abelian qo-groups are isomorphic if the sizes of the generating sets of the group part and the
        monoid part are of the same cardinality.
    \end{proof}

    \begin{definition}
        For a free object $i:M\rightarrow G$, we define the \emph{defect} of $i$ to be $\textup{defect}(i):=\textup{rk}(G)-\textup{rk}(M)$. We say that a free object $i$ is \emph{defective} if $\textup{defect}(i)=\textup{rk}(G)$.
    \end{definition}

    \begin{proposition}
        For a free object $i:M\rightarrow G$, the pair $$(\textup{defect}(i), \textup{rk}(G))$$ determines the abelian qo-group up to isomorphism.
    \end{proposition}

The difference in our approach to free abelian qo-groups from that adopted in \cite{Pigozzi} is that we are focusing on the adjunction between $\mathbf{Set}^m$ and $\mathbf{QAb}^m$, so we obtain free abelian qo-groups with nontrivial order structure.

\section{Monadicity
}\label{S:Monadicity}



\begin{definition}\label{D:monad}
    A \emph{monad} over a category $\mathbf{C}$ is a triple $(T,\eta, \mu)$ consisting of an endofunctor $T:\mathbf{C}\rightarrow \mathbf{C}$, and two natural transformations $\eta: 1\rightarrow T$ and $\mu: T^2\rightarrow T$, called the \emph{unit} and the \emph{multiplication} respectively. The triple must satisfy the identity and associative laws in Figure \ref{F:monadlaws}.
\end{definition}

\begin{figure}[hbt]
    $$\xymatrix{T\ar[r]^{T\eta}\ar@{=}[rd]\ar[d]_{\eta T} & T^2\ar[d]^\mu & T^3\ar[r]^{T\mu}\ar[d]_{\mu T} & T^2\ar[d]^\mu \\ T^2\ar[r]^\mu & T & T^2\ar[r]^\mu & T}$$
    \caption{Identity and Associative Laws of a Monad}\label{F:monadlaws}
\end{figure}

\begin{definition}
    Each monad $(T:\mathbf{C}\rightarrow \mathbf{C},\eta, \mu)$ yields the category $\mathbf{C}^T$ of \emph{Eilenberg-Moore algebras} over that monad. The objects of $\mathbf{C}^T$ are pairs $(x,h)$, where $x$ is a $\mathbf{C}$-object and $h:Tx\rightarrow x$ is a $\mathbf{C}$-morphism satisfying the associative and identity laws of Figure~\ref{F:algebralaws}. The morphisms of $\mathbf{C}^T$ are $\mathbf{C}$-morphisms $f$ that make Figure~\ref{F:structurecompatibility} commute.
\end{definition}

\begin{figure}[hbt]
    $$\xymatrix{x\ar@{=}[rd]\ar[r]^{\eta_x} & Tx\ar[d]^h & T^2x\ar[d]_{Th}\ar[r]^{\mu_x} & Tx\ar[d]^h \\ & x & Tx\ar[r]^h & x}$$
    \caption{Identity and associative laws of a $\mathbf{C}^T$-algebra}\label{F:algebralaws}
\end{figure}
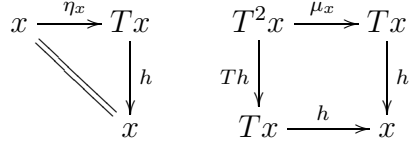

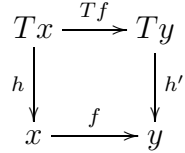
\begin{figure}[hbt]
    $$\xymatrix{Tx\ar[r]^{Tf}\ar[d]_h & Ty\ar[d]^{h'} \\ x\ar[r]^f & y}$$
    \caption{A $\mathbf{C}^T$-Morphism $f$}\label{F:structurecompatibility}
\end{figure}

For any adjunction $(F,U,\eta, \varepsilon)$ with $F:\mathbf{C}\rightarrow \mathbf{D}$, there is a monad $(UF,\eta, U\varepsilon F)$. Furthermore, for this monad, there is an Eilenberg-Moore category $\mathbf{C}^{U\varepsilon F}$. The Eilenberg-Moore category is a terminal object in that there is a unique functor $\overline{G}$ making Figure~\ref{F:terminality} commute.  If $\overline{G}$ is an equivalence, the adjunction is said to be \emph{monadic}. (As noted in \cite{MR1712872}, some authors require an isomorphism between $\mathbf{C}^T$ and $\mathbf{A}$ for an adjunction to qualify as monadic.) If there exists a monadic adjunction with $\mathbf{C}=\mathbf{Set}$, the category $\mathbf{D}$ is said to be \emph{algebraic}. In particular, any variety of algebras is algebraic, motivating the terminology \cite[Cor. IV.4.2.8]{MR1673047}.

\begin{figure}[hbt]

$$\xymatrix{\mathbf{Set}^m\ar@{=}[d]\ar[r] & \mathbf{Set}^{mT}\ar@{<--}[d]^{\overline{G}}\ar[r] & \mathbf{Set}^m\ar@{=}[d] \\
\mathbf{Set}^m\ar[r]^F & \mathbf{QAb}^m\ar[r]^U & \mathbf{Set}^m}$$

\caption{The Eilenberg-Moore comparison}\label{F:terminality}
\end{figure}
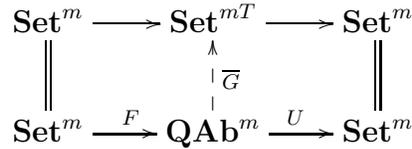

We shall show that the adjunction of Theorem~\ref{T:leftadjoint} is monadic.
When writing elements of a free group or monoid, we adopt the convention of juxtaposition to represent a word, using $1$ for the empty word. Since the initial operation in an abelian qo-group will always be denoted additively, there should be no confusion as to which elements are considered as words and which are considered as elements of the corresponding monoids or groups.

    For an object $(i,h)$ with $i:X'\rightarrow X$ of $\mathbf{Set}^{mUF}$, define 
    $$
    \overline{F}(i:X'\rightarrow X)= (i:X'\rightarrow X) \, .
    $$
    Define an operation $$+_1:X^2\rightarrow X; (x,y)\mapsto h_1(xy)$$ on $X$ and an operation $$+_2:X'^2\rightarrow X'; (x,y)\mapsto h_2(xy)$$ on $X'$. Define $$-: X\rightarrow X; x\mapsto h(x^{-1})\, .$$ Finally, define  $0_1=h_1(1)$ and $0_2=h_2(1)$.

\begin{lemma}\label{L:lem1}
    As defined, $\overline{F}$ maps from the object class of $\mathbf{Set}^{mUF}$ to the object class of $\mathbf{QAb}^m$.
\end{lemma}

\begin{proof}
    We must show, by using the associative and identity laws of the structure map $h$ as expressed in Figure \ref{F:algebralaws}, that $i:X'\rightarrow X$ is indeed a monoid monomorphism from the monoid $X'$ to the underlying monoid of the group $X$. Note that in Figure \ref{F:algebralaws}, $\eta$ is the identity of the
    monad and the unit of adjunction and $\mu= U\varepsilon F$ is the multiplication, and $T=UF$ using the adjunction $(F,U,\eta, \varepsilon)$ from Definition \ref{D:leftadjoint}.

    We see that $+_1$ is associative since $(x+_1 y)+_1z=h_1(h_1(xy)z)$, and the associative law provides that
    $h_1(U_1F_1h_1)=h_1(U_1\varepsilon_1F_1)$, that is to say, $h_1(h_1(xy)h_1(z))=h_1(xyz)=h_1(h_1(x)h_1(yz)$, and the identity law gives that $h_1(z)=z$ for every $z\in X$. Thus $(x+_1 y)+_1z=x+_1 (y+_1z)$. The associative law for $+_2$ follows similarly from the fact that $h_2(U_2F_2h_2)=h_2(U_2\varepsilon_2F_2)$.

    The equality $0_1+_1x=x=x+_10_1$ follows directly from the identity law and the fact that since $F_1$ generates a free abelian group, the letters commute. The identity for $X'$ is similar. We also have that $i(0_2)=0_1$, since $i(h_2(1))=h_1(UFi(1))$ and $Fi(1)=1$. Furthermore, we have that $h_1(UFi(xy))=i(h_2(x,y))$, so $i$ is a monoid homomorphism which is a monomorphism since $i$ is a set monomorphism.

\end{proof}

\begin{lemma}\label{L:lem2}
    The object map as defined in Lemma \ref{L:lem1} has an extension to a morphism part.
\end{lemma}

\begin{proof}
    For $f\in \mathbf{Set}^{mUF}(i,j)$, define $\overline{F}f \in \mathbf{QAb}^m(\overline{F}i, \overline{F}j)$ as follows. Define the individual maps $\overline{F}_1f_1$ and $\overline{F}_2f_2$ from $X$ to $Y$ and from $X'$ to $Y'$ to be the same set maps as in $\mathbf{Set}^mUF$. We need only verify that $f_1$ and $f_2$
    preserve $+_1$ and $+_2$ respectively. Since $f_1\circ h_1=g_1\circ U_1F_1f_1$ we must have $f_1(h_1(x,y))=g_1(f_1(x),f_1(y))$ where $g_1$ is the part of the structure map $j:Y'\rightarrow Y$ which acts on $Y$. Consequently $\overline{F}_2f_2$ is a monoid homomorphism as well.
\end{proof}

\begin{lemma}\label{L:lem3}
    The morphisms $$\overline{F}\colon\mathbf{Set}^{mUF}\rightarrow \mathbf{QAb}^m \mbox{ and  }\ \overline{G}\colon\mathbf{QAb}^m\rightarrow \mathbf{Set}^m$$ are
    mutually inverse.
\end{lemma}

\begin{proof}
    Let $f:i\rightarrow j$ be a morphism in $\mathbf{QAb}^m$. If we apply $\overline{G}$, we obtain the morphism $\overline{G}f:Ui\rightarrow Uj$, where $Ui$ has structure map $h=(h_1,h_2)$ with $h_1(xy)=x+_1y$, $h_2(xy)=x+_2 y$, and $Uj$ has structure map $g=(g_1,g_2)$ with $g_1(xy)=x\cdot_1 y$, $g_2(xy)=x\cdot_2 y$. Then applying $\overline{F}$ to this morphism we obtain $\overline{F}\overline{G}f:Ui\rightarrow Uj$ with binary operations on $Ui$ defined as $x+_1y=h_1(xy)$, $x+_2y=h_2(xy)$, and on $Uj$ as $x\cdot_1y=g_1(xy)$, $x\cdot_2y=g_2(xy)$, which was the morphism that was initially present.

    Let $f:i'\rightarrow j'$ be a morphism in $\mathbf{Set}^m$. As before, we see that the underlying sets of the morphism are unchanged under application of $\overline{F}$ and $\overline{G}$. The operations $+_1$, $+_2$, $\cdot_1$, and $\cdot_2$ are all determined by the action of the structure map on words of length two. Likewise, the structure maps are determined by the values of each of the binary operations previously listed.
\end{proof}

\begin{theorem}\label{T:QAb2Setm}
    The category $\mathbf{QAb}$ of abelian qo-groups is monadic over the category $\mathbf{Set}^m$.
\end{theorem}

\begin{proof}
    We must show that $\mathbf{QAb}$ is equivalent to $\mathbf{Set}^{mT}$ for some endofunctor $T$. Using Lemmas~\ref{L:lem1} to \ref{L:lem3} we see that $\mathbf{Set}^{mUF}$ is isomorphic to $\mathbf{QAb}^m$. In turn, Theorem \ref{T:qabqabi} shows that $\mathbf{QAb}^i$ is isomorphic to $\mathbf{QAb}$. Furthermore Theorem~\ref{T:qabequiv} shows that $\mathbf{QAb}^i$ is equivalent to $\mathbf{QAb}^m$. Therefore we have $\mathbf{Set}^{mUF}\cong \mathbf{QAb}^m \simeq \mathbf{QAb}^i\cong\mathbf{QAb}$ as desired.
\end{proof}

It is readily verified that the adjunction between $\mathbf{Set}^i$ and $\mathbf{QAb}^i$ is also a monadic adjunction. One might hope that there would be a monadic adjunction between $\mathbf{Set}^m$ and $\mathbf{PAb}^m$ defined in a similar way as in Definition~\ref{D:leftadjoint}, since all of the objects in the image of $F$ are isomorphic to objects of $\mathbf{PAb}^m$. That this is not the case, however, is shown by the following example.

\begin{example}
    Let $X'=X=\{x, y\}$ and $i:X'\rightarrow X$ be the identity; also define $h:UFi\rightarrow i$ have $h_1:U_1F_1X\rightarrow X$ where $x$, $xx$, and $yy$ are sent to $x$ and both $y$ and $xy$ are sent to $y$. Then the $UF$-algebra $(i, h)$ becomes isomorphic to $\mathbb{Z}_2$ with nontrivial order, which by Example~\ref{E:finite} is not isomorphic to an object of $\mathbf{PAb}^m$. Therefore $\mathbf{PAb}^m$ is not monadic over $\mathbf{Set}^m$ with the adjunction described.
\end{example}

We end by returning to quasi-ordered groups of divisibility under the partial order of divisibility.

\begin{example}
    Let $\mathbf{Dom}_i$ be the category of integral domains with injective ring homomorphisms. Let $\mathbf{Fld}$ be the category of fields. Then there is a well known functor $Q:\mathbf{Dom}_i\rightarrow \mathbf{Fld}$ which assigns to each integral domain its quotient field. There is another functor $G:\mathbf{Fld}\rightarrow \mathbf{Ab}$ which assigns to each field its multiplicative group of nonzero elements. The quasi-ordered group of divisibility functor is the functor $D:\mathbf{Dom}_i\rightarrow \mathbf{QAb}^m$ taking a domain $R$ to $i: R^*\rightarrow U'GQR$.

    One may also construct the quasi-ordered group of divisibility using the multiplication operation on the quotient field as $+_1$ and the original multiplication of the ring as $+_2$, which is seen to be much more compact in many respects.

    The antisymmetrization of this object will produce the classical group of divisibility.
\end{example}

\bibliographystyle{plain}
\bibliography{EliQOAG}

\end{document}